\documentclass[11pt,reqno]{amsart}
 \usepackage{amssymb,amsfonts,amscd,amsbsy}
\usepackage[mathscr]{eucal}
\usepackage{url}

\newtheorem{theorem}{Theorem}[section]
\newtheorem{lemma}[theorem]{Lemma}
\newtheorem{proposition}[theorem]{Proposition}
\newtheorem{corollary}[theorem]{Corollary}
\theoremstyle{definition}

\numberwithin{equation}{section}

\makeatletter
\def\imod#1{\allowbreak\mkern5mu({\operator@font mod}\,\,#1)}
\makeatother
\allowdisplaybreaks

\begin{document}

\title[Real quadratic double sums]
{Real quadratic double sums} 
 
\author{Jeremy Lovejoy}

\author{Robert Osburn}

\address{CNRS, LIAFA, Universit{\'e} Denis Diderot - Paris 7, Case 7014, 75205 Paris Cedex 13, FRANCE}

\address{School of Mathematical Sciences, University College Dublin, Belfield, Dublin 4, Ireland}

\address{IH{\'E}S, Le Bois-Marie, 35, route de Chartres, F-91440 Bures-sur-Yvette, FRANCE}

\email{lovejoy@liafa.jussieu.fr}

\email{robert.osburn@ucd.ie, osburn@ihes.fr}


\date{\today}

\begin{abstract}
In 1988, Andrews, Dyson and Hickerson initiated the study of $q$-hypergeometric series whose coefficients are dictated by the arithmetic in real quadratic fields.  In this paper, we provide a dozen $q$-hypergeometric double sums which are generating functions for the number of ideals of a given norm in rings of integers of real quadratic fields and prove some related identities. 
\end{abstract}
 
\maketitle

\section{Introduction}
In 1988, Andrews, Dyson and Hickerson \cite{adh} initiated the study of $q$-hypergeometric series whose coefficients are dictated by the arithmetic in real quadratic fields.  They considered a $q$-series from Ramanujan's lost notebook,
\begin{equation} \label{sigmaofq}
\sigma(q) := \sum_{n \geq 0} \frac{q^{\binom{n+1}{2}}}{(-q)_n},
\end{equation}
and proved the Hecke-type identity, 
\begin{equation} \label{sigmaofqid}
\sigma(q) = \sum_{\substack{n \geq 0 \\ -n \leq j \leq n}} (-1)^{n+j}q^{n(3n+1)/2-j^2}(1-q^{2n+1}).
\end{equation}
\noindent Here and throughout we assume that $|q| < 1$ and use the standard $q$-hypergeometric notation,

\begin{equation*}
(a)_n = (a;q)_n = \prod_{k=1}^{n} (1-aq^{k-1}),
\end{equation*}

\noindent valid for $n \in \mathbb{N} \cup \{\infty\}$. 
\noindent Andrews, Dyson and Hickerson then used identity \eqref{sigmaofqid} to relate the coefficients of $\sigma(q)$ to the ring of integers of the real quadratic field $\mathbb{Q}(\sqrt{6})$.  As a consequence, they found that these coefficients satisfy an ``almost'' exact formula, are lacunary and yet, surprisingly, assume all integer values infinitely often.  

Other rare and intriguing examples of $q$-series related to real quadratic fields (predicted to exist by Dyson \cite{dyson}) have been investigated over the years (see \cite{Br-Ka1}, \cite{cflz}, \cite{reallove} and \cite{lovemall}, for example).  The key in each of these studies is the use of Bailey pairs to prove a Hecke-type identity resembling \eqref{sigmaofqid}.  We recall that a \emph{Bailey pair} relative to $a$ is a pair of sequences $(\alpha_n,\beta_n)_{n \geq 0}$ satisfying

\begin{equation} \label{pairdef}
\beta_n = \sum_{k=0}^n \frac{\alpha_k}{(q)_{n-k}(aq)_{n+k}}.
\end{equation} 

For example, Bringmann and Kane \cite{Br-Ka1} discovered the following two Bailey pairs. First, $(a_n,b_n)$ is a Bailey pair relative to $1$, where

\begin{equation} \label{a2n}
a_{2n} = (1-q^{4n})q^{2n^2-2n}\sum_{j=-n}^{n-1}q^{-2j^2-2j},
\end{equation}

\begin{equation} \label{a2n+1}
a_{2n+1} = -(1-q^{4n+2})q^{2n^2}\sum_{j=-n}^{n} q^{-2j^2},
\end{equation}

\noindent and

\begin{equation} \label{bn}
b_n = \frac{(-1)^n(q;q^2)_{n-1}}{(q)_{2n-1}} \chi(\text{$n \neq 0$}).
\end{equation}

\noindent Second, $(\alpha_n,\beta_n)$ is a Bailey pair relative to $q$, where

\begin{equation} \label{alpha2n}
\alpha_{2n} = \frac{1}{1-q}\left(q^{2n^2+2n}\sum_{j=-n}^{n-1}q^{-2j^2-2j} + q^{2n^2}\sum_{j=-n}^{n} q^{-2j^2}\right),
\end{equation}

\begin{equation} \label{alpha2n+1}
\alpha_{2n+1} = -\frac{1}{1-q}\left(q^{2n^2+4n+2}\sum_{j=-n}^{n} q^{-2j^2} + q^{2n^2+2n}\sum_{j=-n-1}^n q^{-2j^2-2j}\right),
\end{equation}

\noindent and 

\begin{equation} \label{betan}
\beta_n = \frac{(-1)^n(q;q^2)_n}{(q)_{2n+1}}.
\end{equation}

Recently, we showed that (\ref{a2n})--(\ref{betan}) are actually special cases of a much more general result (see Theorems 1.1--1.3 in \cite{Lo-Os}).  This led to new Bailey pairs involving indefinite quadratic forms, and we used these new pairs to find many new examples of $q$-hypergeometric double sums which are mock theta functions \cite{Lo-Os}.  In this paper we use these pairs to find many new examples of $q$-hypergeometric double sums which are generating functions for the number of ideals $\frak{a}$ of a given norm $N(\frak{a})$ in the rings of integers $\mathcal{O}_{K}$ of real quadratic fields $K$. Our main results are as follows.  We use the notation ${\sum}^{*}$ to indicate that the sum does not converge in the classical sense, but may be defined as the average of the even and odd partial sums.

 \begin{theorem} \label{quadthm1} Let $K=\mathbb{Q}(\sqrt{2})$. We have that

\begin{equation*} \label{l3}
L_{1}(q) := \sum_{n \geq 1} \sum_{n \geq k \geq 1}  \frac{(q)_{n-1}(-1)^{n+k}q^{\binom{n+1}{2} + \binom{k+1}{2}} }{(q)_{n-k} (q)_{k-1} (1-q^{2k-1})}
\end{equation*}

\noindent satisfies

\begin{equation} \label{L3result}
q^{-17} L_{1}(q^{32}) = \frac{1}{2} \sum_{\substack{\frak{a} \subset \mathcal{O}_{K} \\ N(\frak{a}) \equiv 15 \pmod{32}}} q^{N(\frak{a})},
\end{equation} 

\begin{equation*} \label{l4}
L_{2}(q) := \sum_{n \geq 0} \sum_{n \geq k \geq 0} \frac{(q)_{n}(-1)^{n+k}q^{\binom{n+1}{2} + \binom{k+1}{2}} }{(q)_{n-k} (q)_{k} (1-q^{2k+1})}
\end{equation*}

\noindent satisfies

\begin{equation} \label{L4result}
q^{7} L_{2}(q^{32}) = \frac{1}{2} \sum_{\substack{\frak{a} \subset \mathcal{O}_{K} \\ N(\frak{a}) \equiv 7 \pmod{32}}} q^{N(\frak{a})},
\end{equation}

\begin{equation*}
L_{3}(q) := q \sum_{n \geq 1} \sum_{n \geq k \geq 1}   \frac{(q)_{n-1}(-1)^{n+k}q^{\binom{n+1}{2} + \binom{k}{2}} }{(q)_{n-k} (q)_{k-1} (1-q^{2k-1})} 
\end{equation*}

\noindent satisfies

\begin{equation} \label{L5result}
q^{-33} L_{3}(q^{32}) = \frac{1}{2} \sum_{\substack{\frak{a} \subset \mathcal{O}_{K} \\ N(\frak{a}) \equiv 31 \pmod{32}}} q^{N(\frak{a})},
\end{equation}

\noindent and

\begin{equation*}
L_{4}(q) := -1 + \sum_{n \geq 0} \sum_{n \geq k \geq 0}  \frac{(q)_{n}(-1)^{n+k}q^{\binom{n+1}{2}+ \binom{k}{2}} }{(q)_{n-k} (q)_{k} (1-q^{2k+1})} 
\end{equation*}

\noindent satisfies

\begin{equation} \label{L6result}
q^{-9} L_{4}(q^{32}) = \frac{1}{2} \sum_{\substack{\frak{a} \subset \mathcal{O}_{K} \\ N(\frak{a}) \equiv 23 \pmod{32}}} q^{N(\frak{a})}.
\end{equation}

\end{theorem}

\begin{theorem} \label{quadthm2} Let $L=\mathbb{Q}(\sqrt{3})$. We have that

\begin{equation*} \label{l11}
L_{5}(q) := q \sum_{n \geq 1} \sum_{n \geq k \geq 1}  \frac{(-1)_n (q)_{n-1}(-1)^{n+k}q^{n+k^2-k} }{(q)_{n-k} (q^2; q^2)_{k-1} (1-q^{2k-1})}
\end{equation*}

\noindent satisfies

\begin{equation} \label{L11result}
q^{-2} L_{5}(q^{2}) = 2\sum_{\substack{\frak{a} \subset \mathcal{O}_{L} \\ N(\frak{a}) \equiv 0 \pmod{2} \\ \frak{a} = (x), N(x) < 0}} q^{N(\frak{a})},
\end{equation} 

\begin{equation*} \label{l15}
L_{6}(q) := \sum_{n \geq 1} \sum_{n \geq k \geq 1}  \frac{(-1)_n (q)_{n-1} (-1)^{n+k}q^{n+k^2} }{(q)_{n-k} (q^2;q^2)_{k-1}(1-q^{2k-1})}
\end{equation*}

\noindent satisfies

\begin{equation} \label{L15result}
q^{-1} L_{6}(q^{2}) = 2\sum_{\substack{\frak{a} \subset \mathcal{O}_{L} \\ N(\frak{a}) \equiv 1 \pmod{2} \\ \frak{a} = (x), N(x) < 0}} q^{N(\frak{a})},
\end{equation} 

\begin{equation*} \label{l16}
L_{7}(q) := 2 {\sum\limits_{n \geq 0}}^{*}\sum_{n \geq k \geq 0} \frac{(q^2; q^2)_{n}(-1)^{n+k} q^{k^2 + k}}{(q)_{n-k}(q^2;q^2)_k(1-q^{2k+1})} 
\end{equation*}

\noindent satisfies

\begin{equation} \label{L16result}
qL_{7}(q^{6}) = \sum_{\substack{\frak{a} \subset \mathcal{O}_{L} \\ N(\frak{a}) \equiv 1 \pmod{6}}} q^{N(\frak{a})},
\end{equation} 

\noindent and

\begin{equation*} \label{l17}
L_{8}(q) := -1 + 2 {\sum\limits_{n \geq 0}}^{*}\sum_{n \geq k \geq 0} \frac{(q^2; q^2)_{n}(-1)^{n+k} q^{k^2} }{(q)_{n-k} (q^2; q^2)_k (1-q^{2k+1})} 
\end{equation*}

\noindent satisfies

\begin{equation} \label{L17result}
q^{-2} L_{8}(q^{6}) = \sum_{\substack{\frak{a} \subset \mathcal{O}_{L} \\ N(\frak{a}) \equiv 4 \pmod{6}}} q^{N(\frak{a})}.
\end{equation}

\end{theorem}

\begin{theorem} \label{quadthm3} Let $M=\mathbb{Q}(\sqrt{6})$. We have that

\begin{equation*} \label{l12}
L_{9}(q) := \sum_{n \geq 1}\sum_{n \geq k \geq 1}  \frac{(-1)_n (q)_{n-1}(-1)^{n+k}q^{n + \binom{k+1}{2}} }{(q)_{n-k} (q)_{k-1} (1-q^{2k-1})}
\end{equation*}

\noindent satisfies

\begin{equation} \label{L12result}
q^{-9} L_{9}(q^{16}) = \sum_{\substack{\frak{a} \subset \mathcal{O}_{M} \\ N(\frak{a}) \equiv 7 \pmod{16}}} q^{N(\frak{a})},
\end{equation} 

\begin{equation*} \label{l13}
L_{10}(q) := q \sum_{n \geq 1} \sum_{n \geq k \geq 1}  \frac{(-1)_n (q)_{n-1}(-1)^{n+k}q^{n + \binom{k}{2}} }{(q)_{n-k} (q)_{k-1} (1-q^{2k-1})}
\end{equation*}

\noindent satisfies

\begin{equation} \label{L13result}
q^{-17} L_{10}(q^{16}) = \sum_{\substack{\frak{a} \subset \mathcal{O}_{M} \\ N(\frak{a}) \equiv 15 \pmod{16}}} q^{N(\frak{a})},
\end{equation} 

\begin{equation*} \label{l18}
L_{11}(q) := 2 {\sum\limits_{n \geq 0}}^{*}\sum_{n \geq k \geq 0}  \frac{(q^2; q^2)_{n}(-1)^{n+k}q^{\binom{k+1}{2}} }{(q)_{n-k} (q)_{k} (1-q^{2k+1})}
\end{equation*}

\noindent satisfies

\begin{equation} \label{L18result}
q^{5} L_{11}(q^{48}) = \frac{1}{2} \sum_{\substack{\frak{a} \subset \mathcal{O}_{M} \\ N(\frak{a}) \equiv 5 \pmod{48}}} q^{N(\frak{a})},
\end{equation} 

\noindent and

\begin{equation*} \label{newL12}
L_{12}(q) := -2 + 2 {\sum\limits_{n \geq 0}}^{*}\sum_{n \geq k \geq 0}  \frac{(q^2; q^2)_{n}(-1)^{n+k}q^{\binom{k}{2}} }{(q)_{n-k} (q)_{k} (1-q^{2k+1})}
\end{equation*}

\noindent satisfies

\begin{equation} \label{newL12result}
q^{-19} L_{12}(q^{48}) = \frac{1}{2} \sum_{\substack{\frak{a} \subset \mathcal{O}_{M} \\ N(\frak{a}) \equiv 29 \pmod{48}}} q^{N(\frak{a})}.
\end{equation} 

\end{theorem}

We note that it follows from the above Hecke-type identities and Theorem 1 in \cite{odoni} that all of the series $L_i(q)$ are lacunary.  Exact formulas for the number of elements/ideals in $\mathcal{O}_{K}$, $\mathcal{O}_{L}$ and $\mathcal{O}_{M}$ with prime power norm then imply that their coefficients assume all eligible integer values infinitely often.

With the following corollary we establish identities between some of the real quadratic double sums appearing in Theorems \ref{quadthm1} and \ref{quadthm2} and those which feature prominently in previous related works. Recall the following $q$-series (see (1.2) in \cite{cflz}, (1.9) in \cite{reallove} and Theorems 1.6 and 1.7 in \cite{Br-Ka1}):

\begin{align}
\mathcal{Z}_2(q) &:= \sum_{n \geq 1} \frac{q^n(-q^2;q^2)_{n-1}}{(-q;q^2)_n}, \nonumber \\ 
\mathcal{Z}_3(q) &:= \sum_{n \geq 1} \frac{(-1)^nq^{n^2+n}(q^2;q^2)_{n-1}}{(-q)_{2n}}, \nonumber \\
\mathcal{Z}_4(q) &:= \sum_{n \geq 0} \frac{(-1)^nq^{n^2+n}(q^2;q^2)_n}{(-q)_{2n+1}}, \nonumber \\
\mathcal{Z}_5(q) &:= \sum_{n \geq 1} \frac{(-1)^nq^n(q^2;q^2)_{n}}{(q^{n})_n}. \nonumber
\end{align}

\begin{corollary} \label{realidentities}
We have the following identities: 
\begin{align}
\mathcal{Z}_2(q) &= q^{-2}L_1(q^4) + qL_2(q^4) + q^{-4}L_3(q^4) + q^{-1}L_4(q^4), \label{realidentity1} \\
2\mathcal{Z}_3(q) &= -q^{-2}L_5(q^2) + q^{-1}L_6(q^2), \label{realidentity2} \\
\mathcal{Z}_4(-q)&= L_7(q^2) + q^{-1}L_8(q^2), \label{realidentity3} \\
-2\mathcal{Z}_5(q^2) &= L_6(q). \label{realidentity4}
\end{align}
\end{corollary}

The paper is organized as follows. In Section 2, we first recall some preliminaries on Bailey pairs and key results from \cite{adh} and \cite{Lo-Os}.  In Section 3, we prove Theorems \ref{quadthm1}--\ref{quadthm3} and Corollary \ref{realidentities}. In Section 4, we mention some questions for further study.

\section{Preliminaries}

Before proceeding to the proofs of Theorems \ref{quadthm1}--\ref{quadthm3}, we briefly discuss some preliminaries.  First, the Bailey lemma (see Chapter 3 in \cite{An2}) says that if $(\alpha_n,\beta_n)$  is a Bailey pair relative to $a$, then so is $(\alpha_n',\beta_n')$, where 

\begin{equation} \label{alphaprimedef}
\alpha'_n = \frac{(\rho_1)_n(\rho_2)_n(aq/\rho_1 \rho_2)^n}{(aq/\rho_1)_n(aq/\rho_2)_n}\alpha_n
\end{equation} 

\noindent and

\begin{equation} \label{betaprimedef}
\beta'_n = \sum_{k=0}^n\frac{(\rho_1)_k(\rho_2)_k(aq/\rho_1 \rho_2)_{n-k} (aq/\rho_1 \rho_2)^k}{(aq/\rho_1)_n(aq/\rho_2)_n(q)_{n-k}} \beta_k.
\end{equation}

\noindent The limiting form of the Bailey lemma is found by putting \eqref{alphaprimedef} and \eqref{betaprimedef} into \eqref{pairdef} and letting $n \to \infty$, giving

\begin{equation} \label{limitBailey}
\sum_{n \geq 0} (\rho_1)_n(\rho_2)_n (aq/\rho_1 \rho_2)^n \beta_n = \frac{(aq/\rho_1)_{\infty}(aq/\rho_2)_{\infty}}{(aq)_{\infty}(aq/\rho_1 \rho_2)_{\infty}} \sum_{n \geq 0} \frac{(\rho_1)_n(\rho_2)_n(aq/\rho_1 \rho_2)^n }{(aq/\rho_1)_n(aq/\rho_2)_n}\alpha_n.
\end{equation}

Next, we recall six key Bailey pairs which were established in \cite{Lo-Os}.

\begin{proposition} \label{paircor1} We have the following two Bailey pairs. First, the sequences $(a_n,b_n)$ form a Bailey pair relative to $1$, where

\begin{equation} \label{aevenslater1}
a_{2n} = (1-q^{4n})q^{2n^2-2n+1}\sum_{j=-n}^{n-1}q^{-2j^2},
\end{equation}

\begin{equation} \label{aoddslater1}
a_{2n+1} = -(1-q^{4n+2})q^{2n^2}\sum_{j=-n}^{n}q^{-2j^2-2j},
\end{equation}

\noindent and

\begin{equation} \label{bslater1}
b_n = 
\begin{cases}
0, &\text{if $n=0$},\\
\frac{(-1)^nq^{-n+1}}{(q^2;q^2)_{n-1}(1-q^{2n-1})}, &\text{otherwise}.
\end{cases}
\end{equation}

\noindent Second, the sequences $(\alpha_n,\beta_n)$ form a Bailey pair relative to $q$, where

\begin{equation} \label{aevenslater1q}
\alpha_{2n} = \frac{1}{1-q}\left(q^{2n^2}\sum_{j=-n}^{n}q^{-2j^2-2j} + q^{2n^2+2n+1}\sum_{j=-n}^{n-1}q^{-2j^2}\right), 
\end{equation}

\begin{equation} \label{aoddslater1q}
\alpha_{2n+1} = -\frac{1}{1-q}\left(q^{2n^2+2n+1}\sum_{j=-n-1}^{n}q^{-2j^2} + q^{2n^2+4n+2}\sum_{j=-n}^{n}q^{-2j^2-2j}\right),
\end{equation}

\noindent and

\begin{equation} \label{bslater1q}
\beta_n = \frac{(-1)^nq^{-n}}{(q^2;q^2)_{n}(1-q^{2n+1})}.
\end{equation}

\end{proposition}

\begin{proposition} \label{paircor2} We have the following two Bailey pairs. First, the sequences $(a_n,b_n)$ form a Bailey pair relative to $1$, where

\begin{equation} \label{aevenslater2}
a_{2n} = (1-q^{4n})q^{2n^2-2n}\sum_{j=-n}^{n-1}q^{-4j^2-3j},
\end{equation}

\begin{equation} \label{aoddslater2}
a_{2n+1} = -(1-q^{4n+2})q^{2n^2}\sum_{j=-n}^{n}q^{-4j^2-j},
\end{equation}

\noindent and

\begin{equation} \label{bslater2}
b_n = 
\begin{cases}
0, &\text{if $n=0$},\\
\frac{(-1)^nq^{-\binom{n}{2}}}{(q)_{n-1}(1-q^{2n-1})}, &\text{otherwise}.
\end{cases}
\end{equation}

\noindent Second, the sequences $(\alpha_n,\beta_n)$ form a Bailey pair relative to $q$, where

\begin{equation} \label{aevenslater2q}
\alpha_{2n} = \frac{1}{1-q}\left(q^{2n^2}\sum_{j=-n}^{n}q^{-4j^2-j} + q^{2n^2+2n}\sum_{j=-n}^{n-1}q^{-4j^2-3j}\right), 
\end{equation}

\begin{equation} \label{aoddslater2q}
\alpha_{2n+1} = -\frac{1}{1-q}\left(q^{2n^2+2n}\sum_{j=-n-1}^{n}q^{-4j^2-3j} + q^{2n^2+4n+2}\sum_{j=-n}^{n}q^{-4j^2-j}\right),
\end{equation}

\noindent and

\begin{equation} \label{bslater2q}
\beta_n = \frac{(-1)^nq^{-\binom{n+1}{2}}}{(q)_{n}(1-q^{2n+1})}.
\end{equation}
\end{proposition}

\begin{proposition} \label{paircor3} We have the following two Bailey pairs. First, the sequences $(a_n,b_n)$ form a Bailey pair relative to $1$, where

\begin{equation} \label{aevenslater3}
a_{2n} = (1-q^{4n})q^{2n^2-2n+1}\sum_{j=-n}^{n-1}q^{-4j^2-j}, 
\end{equation}

\begin{equation} \label{aoddslater3}
a_{2n+1} = -(1-q^{4n+2})q^{2n^2}\sum_{j=-n}^{n}q^{-4j^2-3j},
\end{equation}

\noindent and

\begin{equation} \label{bslater3}
b_n = 
\begin{cases}
0, &\text{if $n=0$},\\
\frac{(-1)^nq^{-\binom{n+1}{2}+1}}{(q)_{n-1}(1-q^{2n-1})}, &\text{otherwise}.
\end{cases}
\end{equation}

\noindent Second, the sequences $(\alpha_n,\beta_n)$ form a Bailey pair relative to $q$, where

\begin{equation} \label{aevenslater3q} 
\alpha_{2n} = \frac{1}{1-q}\left(q^{2n^2}\sum_{j=-n}^{n}q^{-4j^2-3j} + q^{2n^2+2n+1}\sum_{j=-n}^{n-1}q^{-4j^2-j}\right), 
\end{equation}

\begin{equation} \label{aoddslater3q}
\alpha_{2n+1} = -\frac{1}{1-q}\left(q^{2n^2+2n+1}\sum_{j=-n-1}^{n}q^{-4j^2-j} + q^{2n^2+4n+2}\sum_{j=-n}^{n}q^{-4j^2-3j}\right),
\end{equation}

\noindent and

\begin{equation} \label{bslater3q}
\beta_n = \frac{(-1)^nq^{-n(n+3)/2}}{(q)_{n}(1-q^{2n+1})}.
\end{equation}
\end{proposition} 

Finally, we record a useful lemma for rewriting Hecke-type sums in terms of rings of integers of real quadratic fields.
\begin{lemma} \cite[Lemma 3]{adh} \label{adhlemma}
Let $(x_1, y_1)$ be the fundamental solution of $x^2 - Dy^2 = 1$, i.e., the solution in which $x_1$ and $y_1$ are minimal positive. 
If $m>0$, then each equivalence class of solutions of $u^2 - Dv^2 = m$ contains a unique $(u,v)$ with $u>0$ and 
\[-\frac{y_1}{x_1+1}u < v\leq \frac{y_1}{x_1+1}u.\]
If $m<0$, the corresponding conditions are $v>0$ and 
\[-\frac{Dy_1}{x_1+1}v < u \leq \frac{Dy_1}{x_1+1}v.\]
\end{lemma}

\section{Proofs of Theorems \ref{quadthm1}--\ref{quadthm3} and Corollary \ref{realidentities}}

We briefly discuss the strategy for proving Theorems \ref{quadthm1}--\ref{quadthm3}. The first step is to make substitutions for $\rho_1$ and $\rho_2$ such that the product on the right-hand side of (\ref{limitBailey}) either simplifies or is eliminated. For example, for Bailey pairs $(\alpha_{n}$, $\beta_{n})$ relative to $a=1$, we can let $\rho_1 \to \infty$, divide both sides by $1-\rho_2$, then let $\rho_2 \to 1$ in (\ref{limitBailey}) to obtain

\begin{equation} \label{a=1}
\sum_{n \geq 0} (-1)^n (q)_{n-1} q^{\frac{n(n+1)}{2}} \beta_{n} = \sum_{n \geq 1} \frac{(-1)^n q^{\frac{n(n+1)}{2}}}{1-q^n} \alpha_{n}.
\end{equation}

\noindent  Alternatively, we can take $\rho_1=-1$ and divide both sides by $1-\rho_2$, then let $\rho_2 \to 1$ in (\ref{limitBailey}) to get

\begin{equation} \label{a=1also}
\sum_{n \geq 0} (-1)_n (q)_{n-1} (-q)^n \beta_{n} = 2\sum_{n \geq 1} \frac{(-q)^n}{1-q^{2n}} \alpha_n.
\end{equation}

\noindent For Bailey pairs $(\alpha_{n}$, $\beta_{n})$ relative to $a=q$, we can let $b \to \infty$ and $c=q$ in (\ref{limitBailey}) to obtain

\begin{equation} \label{a=q}
\sum_{n \geq 0} (-1)^n (q)_n q^{\frac{n(n+1)}{2}} \beta_{n} = (1-q) \sum_{n \geq 0} (-1)^n q^{\frac{n(n+1)}{2}} \alpha_{n}
\end{equation}

\noindent or take $\rho_1=q$ and $\rho_2=-q$ in (\ref{limitBailey}) to obtain

\begin{equation} \label{a=qalso}
\sum_{n \geq 0} (q^2; q^2)_n (-1)^n \beta_n = \frac{1-q}{2} \sum_{n \geq 0} (-1)^n \alpha_n.
\end{equation}

\noindent We then employ the Bailey pairs in Propositions \ref{paircor1}--\ref{paircor3} and the Bailey lemma in \eqref{alphaprimedef} and \eqref{betaprimedef} to obtain a new Bailey pair. Finally, we insert this new pair into one of (\ref{a=1})--(\ref{a=qalso}), express the ``$\alpha_n$'' side in terms of indefinite quadratic forms and appeal to Lemma \ref{adhlemma}.  We now prove Theorems \ref{quadthm1}--\ref{quadthm3}. 

\begin{proof}[Proof of Theorem \ref{quadthm1}]

To prove (\ref{L3result}), we insert (\ref{aevenslater2})--(\ref{bslater2}) into \eqref{alphaprimedef} and \eqref{betaprimedef} with $(a, \rho_1, \rho_2)=(1, \infty, \infty)$, then apply (\ref{a=1}) to get

\begin{equation} \label{L1heck}
\begin{aligned}
L_{1}(q) = \sum_{\substack{n \geq 1 \\ -n \leq j \leq n-1}} q^{8n^2 - n - 4j^2 - 3j} & + q^{8n^2 + n - 4j^2-3j} \\
& + \sum_{\substack{n \geq 0 \\ -n \leq j \leq n}} q^{8n^2 + 7n + 2 - 4j^2 - j} + q^{8n^2 + 9n + 3 - 4j^2 - j}
\end{aligned}
\end{equation}

\noindent and so

\begin{equation} \label{L3sub}
\begin{aligned}
q^{-17} L_{1}(q^{32}) = \sum_{\substack{n \geq 1 \\ -n \leq j \leq n-1}} q^{(16n-1)^2 - 2(8j+3)^2} & + q^{(16n+1)^2 - 2(8j+3)^2} \\
& + \sum_{\substack{n \geq 0 \\ -n \leq j \leq n}} q^{(16n+7)^2 - 2(8j+1)^2} + q^{(16n+9)^2 - 2(8j+1)^2}.
\end{aligned}
\end{equation}

\noindent Thus, 

\begin{multline} \label{L3sub2}
2q^{-17} L_{1}(q^{32}) = \sum_{\substack{n \geq 1 \\ -n \leq j \leq n-1}} q^{(16n-1)^2 - 2(8j+3)^2}  + q^{(16n+1)^2 - 2(8j+3)^2} \\
 + \sum_{\substack{n \geq 0 \\ -n \leq j \leq n}} q^{(16n+7)^2 - 2(8j+1)^2} + q^{(16n+9)^2 - 2(8j+1)^2} \\
 \hskip1in + \sum_{\substack{n \geq 1 \\ -n+1 \leq j \leq n}} q^{(16n-1)^2 - 2(8j-3)^2}  + q^{(16n+1)^2 - 2(8j-3)^2} \\
 + \sum_{\substack{n \geq 0 \\ -n \leq j \leq n}} q^{(16n+7)^2 - 2(8j-1)^2} + q^{(16n+9)^2 - 2(8j-1)^2} 
\end{multline}

\noindent where we have let $j \to -j$ in the second copy of (\ref{L3sub}) to obtain the third and fourth sum in (\ref{L3sub2}). By Lemma \ref{adhlemma} and unique factorization in $\mathcal{O}_{K}$, each ideal $\frak{a}$ can be uniquely written as $\frak{a} = (u + v \sqrt{2})$ with $u >0$ and $-\frac{1}{2} u < v \leq \frac{1}{2}u$. This representation combined with the condition $N(\frak{a}) \equiv 15 \pmod{32}$ is equivalent to either $u \equiv \pm 1 \pmod{16}$, $v \equiv \pm 3 \pmod{8}$ or $u \equiv \pm 7 \pmod{16}$, $v \equiv \pm 1 \pmod{8}$. Comparing this with (\ref{L3sub2}) implies (\ref{L3result}).

To prove (\ref{L4result}), we insert (\ref{aevenslater2q})--(\ref{bslater2q}) into \eqref{alphaprimedef} and \eqref{betaprimedef} with $(a, \rho_1, \rho_2)=(q, \infty, \infty)$, then apply (\ref{a=q}) to get

\begin{equation} \label{L2heck}
\begin{aligned}
L_{2}(q) = \sum_{\substack{n \geq 0 \\ -n \leq j \leq n}} q^{8n^2 + 3n - 4j^2 - j} & + q^{8n^2 + 13n + 5- 4j^2-j} \\
& + \sum_{\substack{n \geq 0 \\ -n-1 \leq j \leq n}} q^{8n^2 + 11n + 3 - 4j^2 -3j} + q^{8n^2 + 21n + 13 - 4j^2 -3j}
\end{aligned}
\end{equation}

\noindent and so

\begin{equation} \label{L4sub}
\begin{aligned}
q^{7} L_{2}(q^{32}) = \sum_{\substack{n \geq 0 \\ -n \leq j \leq n}} q^{(16n+3)^2 - 2(8j+1)^2} & + q^{(16n+13)^2 - 2(8j+1)^2} \\
& + \sum_{\substack{n \geq 0 \\ -n-1 \leq j \leq n}} q^{(16n+21)^2 - 2(8j+3)^2} + q^{(16n+11)^2 - 2(8j+3)^2}.
\end{aligned}
\end{equation}

\noindent Thus,

\begin{multline} \label{L4sub2}
2q^{7} L_{2}(q^{32}) = \sum_{\substack{n \geq 0 \\ -n \leq j \leq n}} q^{(16n+3)^2 - 2(8j+1)^2}  + q^{(16n+13)^2 - 2(8j+1)^2} \\
 + \sum_{\substack{n \geq 0 \\ -n-1 \leq j \leq n}} q^{(16n+21)^2 - 2(8j+3)^2} + q^{(16n+11)^2 - 2(8j+3)^2} \\
 \hskip1in  + \sum_{\substack{n \geq 0 \\ -n \leq j \leq n}} q^{(16n+3)^2 - 2(8j-1)^2} + q^{(16n+13)^2 - 2(8j-1)^2} \\
 + \sum_{\substack{n \geq 0 \\ -n \leq j \leq n+1}} q^{(16n+21)^2 - 2(8j-3)^2} + q^{(16n+11)^2 - 2(8j-3)^2}.
\end{multline}

\noindent  Again, Lemma \ref{adhlemma}, unique factorization and the condition $N(\frak{a}) \equiv 7 \pmod{32}$ imply (\ref{L4result}) after comparing with (\ref{L4sub2}).

For (\ref{L5result}), we insert (\ref{aevenslater3})--(\ref{bslater3}) into \eqref{alphaprimedef} and \eqref{betaprimedef} with $(a, \rho_1, \rho_2)=(1,\infty, \infty)$, then apply (\ref{a=1}) to obtain

\begin{equation} \label{L3heck}
\begin{aligned}
L_3(q) = \sum_{\substack{n \geq 1 \\ -n \leq j \leq n-1}} q^{8n^2 - n +1 - 4j^2 - j} & + q^{8n^2 + n +1 - 4j^2-j} \\
& + \sum_{\substack{n \geq 0 \\ -n \leq j \leq n}} q^{8n^2 + 7n + 2 - 4j^2 - 3j} + q^{8n^2 + 9n + 3 - 4j^2 -3 j} 
\end{aligned}
\end{equation}

\noindent and so

\begin{equation} \label{L5sub}
\begin{aligned}
q^{-33} L_3(q^{32}) = \sum_{\substack{n \geq 1 \\ -n \leq j \leq n-1}} q^{(16n-1)^2 - 2(8j+1)^2} & + q^{(16n+1)^2 - 2(8j+1)^2} \\
& + \sum_{\substack{n \geq 0 \\ -n \leq j \leq n}} q^{(16n+7)^2 - 2(8j+3)^2} + q^{(16n+9)^2 - 2(8j+3)^2}.
\end{aligned}
\end{equation}

\noindent Thus,

\begin{multline} \label{L5sub2}
2q^{-33} L_3(q^{32}) = \sum_{\substack{n \geq 1 \\ -n \leq j \leq n-1}} q^{(16n-1)^2 - 2(8j+1)^2}  + q^{(16n+1)^2 - 2(8j+1)^2} \\
 + \sum_{\substack{n \geq 0 \\ -n \leq j \leq n}} q^{(16n+7)^2 - 2(8j+3)^2} + q^{(16n+9)^2 - 2(8j+3)^2} \\
 \hskip1in +  \sum_{\substack{n \geq 1 \\ -n+1 \leq j \leq n}} q^{(16n-1)^2 - 2(8j-1)^2} + q^{(16n+1)^2 - 2(8j-1)^2} \\
 + \sum_{\substack{n \geq 0 \\ -n \leq j \leq n}} q^{(16n+7)^2 - 2(8j-3)^2} + q^{(16n+9)^2 - 2(8j-3)^2}. 
\end{multline}

\noindent Arguing as above gives (\ref{L5result}).

For (\ref{L6result}), we insert (\ref{aevenslater3q})--(\ref{bslater3q}) into \eqref{alphaprimedef} and \eqref{betaprimedef} with $(a, \rho_1, \rho_2)=(q,\infty, \infty)$, then apply (\ref{a=q}) to obtain

\begin{equation} \label{L4heck}
\begin{aligned}
L_4(q) = -1 + \sum_{\substack{n \geq 0 \\ -n \leq j \leq n}} q^{8n^2 + 3n - 4j^2 - 3j} & + q^{8n^2 + 13n + 5- 4j^2-3j} \\
& + \sum_{\substack{n \geq 0 \\ -n-1 \leq j \leq n}} q^{8n^2 + 11n + 4 - 4j^2 -j} + q^{8n^2 + 21n + 14 - 4j^2 -j}
\end{aligned}
\end{equation}

\noindent and so

\begin{equation} \label{L6sub}
\begin{aligned}
q^{-9} L_{4}(q^{32}) = -q^{-9} + \sum_{\substack{n \geq 0 \\ -n \leq j \leq n}} q^{(16n+3)^2 - 2(8j+3)^2} & + q^{(16n+13)^2 - 2(8j+3)^2} \\
& + \sum_{\substack{n \geq 0 \\ -n-1 \leq j \leq n}} q^{(16n+21)^2 - 2(8j+1)^2} + q^{(16n+11)^2 - 2(8j+1)^2}.
\end{aligned}
\end{equation}

\noindent Slightly modifying the summation limits we obtain
\begin{equation}
\begin{aligned}
q^{-9}L_4(q^{32}) &= \sum_{\substack{n \geq 0 \\ -n \leq j \leq n-1}} q^{(16n+3)^2 -2(8j+3)^2} + \sum_{\substack{n \geq 0 \\ -n-1 \leq j \leq n}} q^{(16n+13)^2 -2(8j+3)^2} \\
&+ \sum_{\substack{n \geq -1 \\ -n-1 \leq j \leq n+1}} q^{(16n+21)^2 -2(8j+1)^2} + \sum_{\substack{n \geq 0 \\ -n \leq j \leq n}} q^{(16n+11)^2 -2(8j+1)^2}.
\end{aligned}
\end{equation}

\noindent Thus, 

\begin{equation} \label{L6sub2}
\begin{aligned}
2q^{-9} L_{4}(q^{32}) &= \sum_{\substack{n \geq 0 \\ -n \leq j \leq n-1}} q^{(16n+3)^2 -2(8j+3)^2} + \sum_{\substack{n \geq 0 \\ -n-1 \leq j \leq n}} q^{(16n+13)^2 -2(8j+3)^2} \\
&+ \sum_{\substack{n \geq -1 \\ -n-1 \leq j \leq n+1}} q^{(16n+21)^2 -2(8j+1)^2} + \sum_{\substack{n \geq 0 \\ -n \leq j \leq n}} q^{(16n+11)^2 -2(8j+1)^2} \\
&+ \sum_{\substack{n \geq 0 \\ -n+1 \leq j \leq n}} q^{(16n+3)^2 -2(8j-3)^2} + \sum_{\substack{n \geq 0 \\ -n \leq j \leq n+1}} q^{(16n+13)^2 -2(8j-3)^2} \\
&+ \sum_{\substack{n \geq -1 \\ -n-1 \leq j \leq n+1}} q^{(16n+21)^2 -2(8j-1)^2} + \sum_{\substack{n \geq 0 \\ -n \leq j \leq n}} q^{(16n+11)^2 -2(8j-1)^2}.
\end{aligned}
\end{equation}

\noindent Arguing as before gives (\ref{L6result}).
\end{proof}

\begin{proof}[Proof of Theorem \ref{quadthm2}]

For (\ref{L11result}), insert (\ref{aevenslater1})--(\ref{bslater1}) into \eqref{alphaprimedef} and \eqref{betaprimedef} with $(a,\rho_1, \rho_2)=(1, \infty, \infty)$, then apply (\ref{a=1also}) to obtain

\begin{equation} \label{L5heck}
L_{5}(q) = 2\sum_{\substack{n \geq 1 \\ -n \leq j \leq n-1}} q^{6n^2 + 1 - 2j^2} + 2\sum_{\substack{n \geq 0 \\ -n \leq j \leq n}} q^{6n^2 + 6n + 2- 2j^2 - 2j}
\end{equation}

\noindent and so

\begin{equation} \label{L11sub}
q^{-2} L_{5}(q^2) =  2\sum_{\substack{n \geq 1 \\ -n \leq j \leq n-1}} q^{3(2n)^2 - (2j)^2} + 2\sum_{\substack{n \geq 0 \\ -n \leq j \leq n}} q^{3(2n+1)^2 - (2j+1)^2}.
\end{equation}

\noindent By Lemma \ref{adhlemma} and unique factorization in $\mathcal{O}_{L}$, each (principal) ideal $\frak{a}$ generated by an element of negative norm can be uniquely written as $\frak{a}= (u + v \sqrt{3})$ with $v >0$ and $-v < u \leq v$. This representation combined with the condition $N(\frak{a}) \equiv 0 \pmod{2}$ is equivalent to either $u \equiv 0 \pmod{2}$, $v \equiv 0 \pmod{2}$ or $u \equiv 1 \pmod{2}$, $v \equiv 1 \pmod{2}$.  Comparing this with (\ref{L11sub}) implies (\ref{L11result}).

For (\ref{L15result}), insert (\ref{a2n})--(\ref{bn}) into \eqref{alphaprimedef} and \eqref{betaprimedef} with $(a,\rho_1, \rho_2)=(q, \infty, \infty)$, then apply (\ref{a=1also}) to get

\begin{equation} \label{L6heck}
L_{6}(q) = 2\sum_{\substack{n \geq 1 \\ -n \leq j \leq n-1}} q^{6n^2 - 2j^2 - 2j} + 2\sum_{\substack{n \geq 0 \\ -n \leq j \leq n}} q^{6n^2 + 6n + 1 - 2j^2}
\end{equation}

\noindent and so

\begin{equation} \label{L15sub}
q^{-1} L_{6}(q^2) =  2\sum_{\substack{n \geq 1 \\ -n \leq j \leq n-1}} q^{3(2n)^2 - (2j+1)^2} + 2\sum_{\substack{n \geq 0 \\ -n \leq j \leq n}} q^{3(2n+1)^2 - (2j)^2}.
\end{equation}

\noindent By Lemma \ref{adhlemma} and unique factorization in $\mathcal{O}_{L}$, each principal ideal $\frak{a}$ generated by an element of negative norm can be uniquely written as 
$\frak{a}= (u + v \sqrt{3})$ with $v > 0$, $-v < u \leq v$. Arguing as usual gives (\ref{L15result}).

For (\ref{L16result}), insert (\ref{alpha2n})--(\ref{betan}) into \eqref{alphaprimedef} and \eqref{betaprimedef} with $(a,\rho_1, \rho_2)=(q, \infty, \infty)$, then apply (\ref{a=qalso}) to get

\begin{equation} \label{L7heck}
\begin{aligned}
L_{7}(q) =  \sum_{\substack{n \geq 0 \\ -n-1 \leq j \leq n}} q^{6n^2 + 16n + 10 - 2j^2 - 2j} & + q^{6n^2 + 8n +2 - 2j^2 - 2j} \\
& + \sum_{\substack{n \geq 0 \\ -n \leq j \leq n}} q^{6n^2 + 2n - 2j^2} + q^{6n^2 + 10n + 4 - 2j^2}
\end{aligned}
\end{equation}

\noindent and so 

\begin{equation} \label{L16sub}
\begin{aligned}
q L_{7}(q^6) = \sum_{\substack{n \geq 0 \\ -n-1 \leq j \leq n}} q^{(6n+8)^2 - 3(2j+1)^2} & + q^{(6n+4)^2 - 3(2j+1)^2} \\
& + \sum_{\substack{n \geq 0 \\ -n \leq j \leq n}} q^{(6n+1)^2 - 3(2j)^2} + q^{(6n+5)^2 - 3(2j)^2}.
\end{aligned}
\end{equation}

\noindent Arguing as usual gives (\ref{L16result}).

For (\ref{L17result}), insert (\ref{aevenslater1q})--(\ref{bslater1q}) into \eqref{alphaprimedef} and \eqref{betaprimedef} with $(a,\rho_1, \rho_2)=(q,\infty, \infty)$, then apply (\ref{a=qalso}) to obtain

\begin{equation} \label{L8heck}
\begin{aligned}
L_{8}(q) =  -1 + \sum_{\substack{n \geq 0 \\ -n \leq j \leq n}} q^{6n^2 + 2n - 2j^2 - 2j} & + q^{6n^2 + 10n + 4 - 2j^2 -2j} \\
& + \sum_{\substack{n \geq 0 \\ -n-1 \leq j \leq n}} q^{6n^2 + 16n + 11 - 2j^2} + q^{6n^2 + 8n + 3 -2j^2}
\end{aligned}
\end{equation}

\noindent and so

\begin{equation*} \label{L17sub}
\begin{aligned}
q^{-2}L_{8}(q^6) =  q^{-2} + \sum_{\substack{n \geq 0 \\ -n \leq j \leq n}} q^{(6n+1)^2 - 3(2j+1)^2} & + q^{(6n+5)^2 - 3(2j+1)^2} \\
& + \sum_{\substack{n \geq 0 \\ -n-1 \leq j \leq n}} q^{(6n+8)^2 - 3(2j)^2} + q^{(6n+4)^2 - 3(2j)^2}.
\end{aligned}
\end{equation*}

\noindent Slightly modifying the summation bounds and simplifying gives
\begin{equation*}
\begin{aligned}
q^{-2}L_8(q^6) &= \sum_{\substack{n \geq 0 \\ -n \leq j \leq n-1}} q^{(6n+1)^2 - 3(2j+1)^2} + \sum_{\substack{n \geq 0 \\ -n-1 \leq j \leq n}} q^{(6n+5)^2 - 3(2j+1)^2}  \\ 
&+ \sum_{\substack{n \geq -1 \\ -n-1 \leq j \leq n+1}} q^{(6n+8)^2 - 3(2j)^2} + \sum_{\substack{n \geq 0 \\ -n \leq j \leq n}} q^{(6n+4)^2 - 3(2j)^2}.
\end{aligned}
\end{equation*}
Arguing as usual gives (\ref{L17result}). 
\end{proof}

\begin{proof}[Proof of Theorem \ref{quadthm3}]

For (\ref{L12result}), we insert (\ref{aevenslater2})--(\ref{bslater2}) into \eqref{alphaprimedef} and \eqref{betaprimedef} with $(a,\rho_1, \rho_2)=(1,\infty, \infty)$, then apply (\ref{a=1also}) to get

\begin{equation} \label{L9heck}
L_{9}(q) = 2\sum_{\substack{n \geq 1 \\ -n \leq j \leq n-1}} q^{6n^2 - 4j^2 -3j} + 2\sum_{\substack{n \geq 0 \\ -n \leq j \leq n}} q^{6n^2 + 6n + 2- 4j^2 - j}
\end{equation}

\noindent and so

\begin{equation} \label{L12sub}
\begin{aligned}
q^{-9} L_{9}(q^{16}) &=  \sum_{\substack{n \geq 1 \\ -n \leq j \leq n-1}} q^{6(4n)^2 - (8j+3)^2} + \sum_{\substack{n \geq 0 \\ -n \leq j \leq n}} q^{6(4n+2)^2 - (8j+1)^2} \\
&+ \sum_{\substack{n \geq 1 \\ -n+1 \leq j \leq n}} q^{6(4n)^2 - (8j-3)^2} + \sum_{\substack{n \geq 0 \\ -n \leq j \leq n}} q^{6(4n+2)^2 - (8j-1)^2}. 
\end{aligned}
\end{equation}

By Lemma \ref{adhlemma} and unique factorization in $\mathcal{O}_{M}$, each ideal $\frak{a}$ with $N(\frak{a}) \equiv 7 \pmod{16}$ can be uniquely written as $\frak{a}= (u + v \sqrt{6})$ with $v >0$ and $-2v < u \leq 2v$. Arguing as usual gives (\ref{L12result}).

For (\ref{L13result}), we insert (\ref{aevenslater3})--(\ref{bslater3}) into \eqref{alphaprimedef} and \eqref{betaprimedef} with $(a,\rho_1, \rho_2)=(1,\infty, \infty)$, then apply (\ref{a=1also}) to get

\begin{equation} \label{L10heck}
L_{10}(q) = 2\sum_{\substack{n \geq 1 \\ -n \leq j \leq n-1}} q^{6n^2 + 1- 4j^2 -j} + 2\sum_{\substack{n \geq 0 \\ -n \leq j \leq n}} q^{6n^2 + 6n + 2- 4j^2 -3 j}
\end{equation}

\noindent and so

\begin{equation} \label{L13sub}
\begin{aligned}
q^{-17} L_{10}(q^{16}) &=  \sum_{\substack{n \geq 1 \\ -n \leq j \leq n-1}} q^{6(4n)^2 - (8j+1)^2} + \sum_{\substack{n \geq 0 \\ -n \leq j \leq n}} q^{6(4n+2)^2 - (8j+3)^2} \\ 
&+ \sum_{\substack{n \geq 1 \\ -n+1 \leq j \leq n}} q^{6(4n)^2 - (8j-1)^2} + \sum_{\substack{n \geq 0 \\ -n \leq j \leq n}} q^{6(4n+2)^2 - (8j-3)^2}.
\end{aligned}
\end{equation}

\noindent Arguing as usual gives (\ref{L13result}).

For (\ref{L18result}), we insert (\ref{aevenslater2q})--(\ref{bslater2q}) into \eqref{alphaprimedef} and \eqref{betaprimedef} with $(a,\rho_1, \rho_2)=(q,\infty, \infty)$, then apply (\ref{a=qalso}) to obtain

\begin{equation} \label{L11heck}
\begin{aligned}
L_{11}(q) = \sum_{\substack{n \geq 0 \\ -n \leq j \leq n}} q^{6n^2 + 2n - 4j^2 -j} & + q^{6n^2 + 10n + 4 - 4j^2 - j} \\
& + \sum_{\substack{n \geq 0 \\ -n-1 \leq j \leq n}} q^{6n^2 + 16n + 10 - 4j^2 - 3j} + q^{6n^2 + 8n + 2 - 4j^2 - 3j}
\end{aligned}
\end{equation}

\noindent and so

\begin{equation} \label{L18sub}
\begin{aligned}
q^{10} L_{11}(q^{96}) =  \sum_{\substack{n \geq 0 \\ -n \leq j \leq n}} q^{(24n+4)^2 - 6(8j+1)^2} & + q^{(24n + 20)^2 - 6(8j+1)^2} \\
& + \sum_{\substack{n \geq 0 \\ -n-1 \leq j \leq n}} q^{(24n+32)^2 - 6(8j+3)^2} + q^{(24n+16)^2 - 6(8j+3)^2}.
\end{aligned}
\end{equation}

\noindent Thus, 

\begin{multline} \label{L18sub2}
2q^{10} L_{11}(q^{96}) =  \sum_{\substack{n \geq 0 \\ -n \leq j \leq n}} q^{(24n+4)^2 - 6(8j+1)^2}  + q^{(24n + 20)^2 - 6(8j+1)^2} \\
 + \sum_{\substack{n \geq 0 \\ -n-1 \leq j \leq n}} q^{(24n+32)^2 - 6(8j+3)^2} + q^{(24n+16)^2 - 6(8j+3)^2} \\
 \hskip1in +  \sum_{\substack{n \geq 0 \\ -n \leq j \leq n}} q^{(24n+4)^2 - 6(8j-1)^2} + q^{(24n + 20)^2 - 6(8j-1)^2} \\
 + \sum_{\substack{n \geq 0 \\ -n \leq j \leq n+1}} q^{(24n+32)^2 - 6(8j-3)^2} + q^{(24n+16)^2 - 6(8j-3)^2}.
\end{multline} 

\noindent Arguing as usual gives 
\begin{equation} 
q^{10} L_{11}(q^{96}) = \frac{1}{2} \sum_{\substack{\frak{a} \subset \mathcal{O}_{M} \\ N(\frak{a}) \equiv 10 \pmod{96}}} q^{N(\frak{a})},
\end{equation} 
and dividing by the unique ideal $(2+\sqrt{6})$ in $\mathcal{O}_M$ of norm $2$ gives (\ref{L18result}). 

For (\ref{newL12result}), we insert (\ref{aevenslater3q})--(\ref{bslater3q}) into \eqref{alphaprimedef} and \eqref{betaprimedef} with $(a,\rho_1, \rho_2)=(q,\infty, \infty)$, then apply (\ref{a=qalso}) to obtain

\begin{equation} \label{L12heck}
\begin{aligned}
L_{12}(q) = -2 + \sum_{\substack{n \geq 0 \\ -n \leq j \leq n}} q^{6n^2 + 2n - 4j^2 - 3j} & + q^{6n^2 + 10n + 4 - 4j^2 - 3j} \\
& + \sum_{\substack{n \geq 0 \\ -n-1 \leq j \leq n}} q^{6n^2 + 16n + 11 - 4j^2 - j} + q^{6n^2 + 8n + 3 - 4j^2 - j}.
\end{aligned}
\end{equation}

\noindent Thus,

\begin{equation*}
\begin{aligned}
q^{-38} L_{12}(q^{96}) = -2q^{-38} & + \sum_{\substack{n \geq 0 \\ -n \leq j \leq n}} q^{(24n+4)^2 - 6(8j+3)^2} + q^{(24n + 20)^2 - 6(8j+3)^2} \\
& + \sum_{\substack{n \geq 0 \\ -n-1 \leq j \leq n}} q^{(24n+32)^2 - 6(8j+1)^2} + q^{(24n+16)^2 - 6(8j+1)^2}.
\end{aligned}
\end{equation*}

\noindent Slightly modifying the summation bounds and simplifying gives

\begin{equation}
\begin{aligned}
q^{-38} L_{12}(q^{96}) & = \sum_{\substack{n \geq 1 \\ -n \leq j \leq n-1}} q^{(24n+4)^2 - 6(8j+3)^2} +  \sum_{\substack{n \geq 0 \\ -n-1 \leq j \leq n}} q^{(24n + 20)^2 - 6(8j+3)^2} \\
& +  \sum_{\substack{n \geq -1 \\ -n-1 \leq j \leq n+1}} q^{(24n+32)^2 - 6(8j+1)^2} + \sum_{\substack{n \geq 0 \\ -n \leq j \leq n}} q^{(24n+16)^2 - 6(8j+1)^2}.
\end{aligned}
\end{equation}

\noindent Thus, 

\begin{equation}
\begin{aligned}
2q^{-38} L_{12}(q^{96}) & =  \sum_{\substack{n \geq 1 \\ -n \leq j \leq n-1}} q^{(24n+4)^2 - 6(8j+3)^2} +  \sum_{\substack{n \geq 0 \\ -n-1 \leq j \leq n}} q^{(24n + 20)^2 - 6(8j+3)^2} \\
& + \sum_{\substack{n \geq 1 \\ -n+1 \leq j \leq n}} q^{(24n+4)^2 - 6(8j-3)^2} +  \sum_{\substack{n \geq 0 \\ -n \leq j \leq n+1}} q^{(24n + 20)^2 - 6(8j-3)^2} \\
& +  \sum_{\substack{n \geq -1 \\ -n-1 \leq j \leq n+1}} q^{(24n+32)^2 - 6(8j+1)^2} + \sum_{\substack{n \geq 0 \\ -n \leq j \leq n}} q^{(24n+16)^2 - 6(8j+1)^2} \\
& +  \sum_{\substack{n \geq -1 \\ -n-1 \leq j \leq n+1}} q^{(24n+32)^2 - 6(8j-1)^2} + \sum_{\substack{n \geq 0 \\ -n \leq j \leq n}} q^{(24n+16)^2 - 6(8j-1)^2}.
\end{aligned}
\end{equation}

\noindent Arguing as usual gives (\ref{newL12result}). 
\end{proof}

\begin{proof}[Proof of Corollary \ref{realidentities}]
By Theorem 3.3 of \cite{cflz} we have that 

\begin{equation} \label{realid1}
q^{-1} \mathcal{Z}_2(q^8) = \frac{1}{2} \sum_{\substack{\frak{a} \subset \mathcal{O}_{K} \\ N(\frak{a}) \equiv 7 \pmod{8}}} q^{N(\frak{a})}.
\end{equation}

Comparing (\ref{realid1}) with equations (\ref{L3result})--(\ref{L6result}) gives \eqref{realidentity1}. Next, Theorem 1.2 of \cite{reallove} is equivalent to 

\begin{equation} \label{realid2}
\mathcal{Z}_3(q)  = -\sum_{\substack{\frak{a} \subset \mathcal{O}_{L} \\ \frak{a} = (x), N(x) < 0}} (-1)^{N(\frak{a})}q^{N(\frak{a})}.
\end{equation}

\noindent One compares (\ref{realid2}) with equations (\ref{L11result}) and (\ref{L15result}) to obtain \eqref{realidentity2}.  In Theorem 1.7 of \cite{Br-Ka1} it is shown that

\begin{equation} \label{realid3}
q\mathcal{Z}_4(q^3) = -\sum_{\substack{\frak{a} \subset \mathcal{O}_{L} \\ N(\frak{a}) \equiv 1 \pmod{3}}} (-1)^{N(\frak{a})} q^{N(\frak{a})}.
\end{equation}

\noindent Comparing (\ref{realid3}) with equations \eqref{L16result} and \eqref{L17result} gives \eqref{realidentity3}. Finally, in Theorem 1.6 of \cite{Br-Ka1} it is shown that
\begin{equation} \label{realid4}
q^{-1}\mathcal{Z}_5(q^4) = -\sum_{\substack{\frak{a} \subset \mathcal{O}_{L} \\ N(\frak{a}) \equiv 3 \pmod{4}}} q^{N(\frak{a})}.  
\end{equation}
The sum on the right-hand side of \eqref{realid4} is identical to the sum on the right-hand side of \eqref{L15result}, giving \eqref{realidentity4}.
\end{proof}

\section{Questions for further study}
The series $\sigma(q)$ has been related to Maass waveforms by Cohen \cite{Co1} and to quantum modular forms by Zagier \cite{Zag2}.  The relation of the series $L_i(q)$ to Maass waveforms could be made precise using work of Zwegers \cite{zwegers}, but it is unclear whether there is a relation to quantum modular forms.   This is worth investigating.  The combinatorics of these series is also worth pursuing.  Do they have an elegant partition-theoretic interpretation?  Is there a natural explanation for the positivity of their coefficients?

\section*{Acknowledgements} The second author would like to thank the Institut des Hautes {\'E}tudes Scientifiques for their support during the completion of this paper. This material is based upon work supported by the National Science Foundation under Grant No. 1002477.

\end{document}